\documentclass[a4paper,12pt]{article}

\usepackage{amssymb,amsmath}
\usepackage[abbrev]{amsrefs}
\usepackage{xy}
\xyoption{all}
\numberwithin{equation}{section}
\usepackage{amsthm}
\swapnumbers
\newtheorem{theorem}{Theorem}[section]
\newtheorem{lemma}[theorem]{Lemma}
\newtheorem{corollary}[theorem]{Corollary}
\newtheorem{exm}[theorem]{Example}
\newenvironment{example}{\begin{exm}\rm}{\end{exm}}
\newcommand\V{\bigvee}
\newcommand\Max{\operatorname{Max}}
\newcommand\ie{i.e.}
\newcommand\st{\mid}
\newcommand\cf{\textrm{cf.}}
\newcommand\topology{\operatorname{\Omega}}
\newcommand\Loc{\textit{Loc}}
\newcommand\Frm{\textit{Frm}}
\newcommand\opp[1]{{#1}^{\textrm{op}}}
\newcommand\CC{\mathbb{C}}
\newcommand\QS{\textit{QSp}}
\newcommand\Qu{\textit{Qu}}
\newcommand\qtimes{\operatorname{\hat{\otimes}}}
\newcommand\downsegment{{\downarrow}}
\newcommand\shortiff{{\Leftrightarrow}}

\begin{document}

\title{Open maps of involutive quantales\footnote{Work funded by FCT/Portugal through project PEst-OE/EEI/LA0009/2013 and by COST (European Cooperation in Science and Technology) through COST Action MP1405 QSPACE.}}

\date{~}

\author{Pedro Resende}

\maketitle

\begin{abstract}
By a map $p:Q\to X$ of involutive quantales is meant a homomorphism $p^*:X\to Q$. Calling a map $p$ \emph{weakly open} if $p^*$ has a left adjoint $p_!$ which satisfies the Frobenius reciprocity condition (\ie, $p_!$ is a homomorphism of $X$-modules), we say that $p$ is \emph{open} if it is stably weakly open. We also study a two-sided version, FR2, of the Frobenius reciprocity condition, and show that the weakly open surjections that satisfy FR2 are open. Maps of the latter kind arise in the study of Fell bundles on groupoids.
\\
\vspace*{-2mm}~\\
\textit{Keywords:} Involutive quantales, open maps, Frobenius reciprocity condition.\\
\vspace*{-2mm}~\\
2010 \textit{Mathematics Subject
Classification}: 06F07, 18B30
\end{abstract}

\setcounter{tocdepth}{1}
\tableofcontents

\section{Introduction}\label{sec:introduction}

Quantales, in particular involutive quantales, are generalizations of locales which often can be regarded as generalized spaces. For instance, there are several notions of point of a quantale \cites{K02,MP1,MR}, 
and there are correspondences between quantales and other types of generalized space such as Grothendieck toposes \cites{HeymansGrQu,HenryPhD,GSQS,HS2}, C*-algebras \cites{Rosicky,MP1,MP2,BRB,KR}, or groupoids \cites{Re07,PR12}. To some extent the geometric meaning of quantales is borrowed from such structures, and also from the functoriality of the correspondences, which in turn depends on which notion of morphism of quantales we adopt.

No choice of morphism works for all purposes.
For instance, inverse quantal frames~\cite{Re07} form a bicategory which is biequivalent to that of \'etale groupoids~\cite{Re15}, and this suggests that ``geometric maps'' of these quantales can be defined in terms of quantale bimodules so as to correspond to those groupoid bi-actions which in turn yield geometric morphisms of \'etendues \cites{Moer90,Moer87,Bunge}. (A similar notion of morphism of groupoids is relevant in the context of strong Morita equivalence for C*-algebras~\cites{MRW87,La01,Mr99}.) Another more immediate (and in general different) notion is based on homomorphisms of involutive quantales, which correspond, for groupoids, to the algebraic morphisms of \cites{BS05, Bun08}. See \cite{Re15}. Homomorphisms are also natural when relating C*-algebras and quantales, since $*$-homomorphisms of C*-algebras translate functorially to involutive homomorphisms between their quantales~\cite{MP1}. And another C*-algebra-related context, in fact the one which motivates the present paper, is that of Fell bundles on groupoids~\cite{Kumjian98}, where again quantale homomorphisms arise naturally~\cite{QFB}.

Here we shall be concerned with general involutive homomorphisms of involutive quantales or, rather, with their formal duals as in locale theory~\cite{stonespaces}: by a \emph{(continuous) map} of involutive quantales $p:Q\to X$ will be meant an involutive homomorphism $p^*:X\to Q$, \ie, a mapping that satisfies the following conditions for all $x,y\in X$ and all families $(x_i)$ in $X$:
\begin{eqnarray*}
p^*(xy) &=& p^*(x)p^*(y)\\
p^*(x^*)&=& \bigl(p^*(x)\bigr)^*\\
p^*\bigl(\V_i x_i\bigr)&=&\V_i p^*(x_i)
\end{eqnarray*}
This is referred to as the \emph{inverse image homomorphism} of $p$. Following the terminology for locales in~\cite{RV}, if $p^*$ has a (necessarily involution preserving) left adjoint $p_!$ we say that $p$ is \emph{semiopen}, and refer to $p_!$ as the \emph{direct image homomorphism} of $p$.

The main purpose of this short paper is to examine conditions under which a semiopen map of involutive quantales can be regarded as being open. This is not for the sake of generalization per se, but rather because examples of open-like maps arise from some Fell bundles on groupoids (see section~\ref{sec:frob}), and it is worth understanding their properties in particular as regards stability under pullbacks.

\section{Preliminaries}\label{sec:subspaces}

This section is mostly for fixing terminology and notation.

\paragraph{Quantic subspaces.} We shall denote by $\Qu$ the usual category of involutive quantales whose arrows are the involutive homomorphisms, and by $\QS$ the opposite category $\opp\Qu$, whose arrows are the maps. We refer to $\QS$ as the category of \emph{quantic spaces}, in order to disambiguate our terminology, for instance when referring to subobjects: whereas by an \emph{involutive subquantale} is usually meant a subobject of an involutive quantale in the algebraic sense, in this paper a \emph{quantic subspace} of an involutive quantale $Q$ is defined to be an equivalence class of regular monomorphisms $m:S\to Q$ in $\QS$; that is, an equivalence class of maps $m$ such that $m^*$ is a surjective homomorphism, where two such maps $m:S\to Q$ and $m':S'\to Q$ are equivalent if there is an isomorphism $S\cong S'$ commuting with $m$ and $m'$.\footnote{We note that the terminology `quantic space versus quantale' is consistent with the terminology `space versus locale' of Joyal and Tierney~\cite{JT}, albeit not with the more common terminology `locale versus frame'. An alternative would be to use `quantale' in the topological sense and, say, `quantic frame' in the algebraic one. However, the ensuing meaning of  `subquantale' would go against the usual terminology in quantale theory, and, moreover, `quantic frame' is misleadingly close to `quantal frame', which has been consistently used by more than one author in order to refer to a quantale whose order satisfies the locale distributivity law.}

A quantic subspace of $Q$ can also be identified with an \emph{involutive quantic nucleus} on $Q$, \ie, a closure operator on $Q$ that satisfies $j(a)j(b)\le j(ab)$ and $j(a^*)=j(a)^*$ for all $a,b\in Q$, as we now explain.
Let $Q$ be an involutive quantale, and let $j:Q\to Q$ be an involutive quantic nucleus.
The set of closed elements
\[
Q_j=j(Q)=\{a\in Q\st j(a)=a\}
\]
is closed under meets and involution, and it is an involutive quantale with multiplication defined by $(a,b)\mapsto j(ab)$ and the join of each family $(a_i)$ in $Q_j$ being $j\bigl(\V_i a_i\bigr)$. Moreover, $j:Q\to Q_j$ is a surjective homomorphism of involutive quantales, and, up to isomorphisms, every surjective homomorphism of involutive quantales arises like this. For a proof of this (for non involutive quantales) see \cite{Rosenthal1}. Hence, the involutive quantic nuclei on $Q$ are in bijective correspondence with quantic subspaces of $Q$: given a regular monomorphism $m:S\to Q$ in $\QS$ we have $S\cong Q_j$, where $j=m_*\circ m^*$ and $m_*$ is the right adjoint of $m^*$.

\paragraph{Quantic subspaces presented by relations.} Let $Q$ be an involutive quantale. If $R\subseteq Q\times Q$ is a binary relation on $Q$, the least (in the pointwise order) involutive quantic nucleus $j$ on $Q$ such that $j(r)=j(s)$ for all $(r,s)\in R$ is denoted by $j_R$, and
\[
\iota_R:Q_{j_R}\to Q
\]
is the regular monomorphism defined by $\iota_R^*=j_R$.
Any regular monomomorphism \[\xi:S\to Q\] such that $\xi^*(r)=\xi^*(s)$ for all $(r,s)\in R$ factors (uniquely) through $\iota_R$:
\[
\xymatrix{
S\ar[rrrr]^{\xi}\ar[drr]&&&&Q\\
&&Q_{j_R}\ar[rru]_{\iota_R}
}
\]
This provides us with the construction of equalizers in $\QS$: the equalizer of two maps $f,g:Q\to X$ is the quantic subspace $\iota_R:Q_{j_R}\to Q$ for $R=\{(f^*(x),g^*(x))\st x\in X\}$.

The following facts are easily derived from \cite{JT} and are useful in calculations in order to translate the universal property of $Q_{j_R}$ into sup-lattices:

\begin{lemma}\label{relations}
Let $Q$ be an involutive quantale, let $R\subseteq Q\times Q$ be a binary relation on $Q$, and let $\widetilde R\subseteq Q\times Q$ be the least binary relation on $Q$ such that the following conditions hold:
\begin{itemize}
\item $R\subseteq\widetilde R$,
\item $(r^*,s^*)\in\widetilde R$ for all $(r,s)\in\widetilde R$,
\item $(ar,as)\in\widetilde R$ for all $(r,s)\in\widetilde R$ and all $a\in Q$.
\end{itemize}
The following equivalent conditions hold:
\begin{enumerate}
\item The quantic nucleus $j_R$ coincides with the least closure operator $j$ on $Q$ such that $j(r)=j(s)$ for all $(r,s)\in\widetilde R$.
\item $Q_{j_R}$ consists of those $\alpha\in Q$ such that for all $(r,s)\in \widetilde R$ the following condition holds:
\begin{eqnarray*}
r\le \alpha&\iff& s\le \alpha\;.
\end{eqnarray*}
\item Any sup-lattice homomorphism $h:Q\to L$, where $L$ is a sup-lattice, factors through the quotient homomorphism $j_R:Q\to Q_{j_R}$ if and only if for all $(r,s)\in \widetilde R$ we have
$h(r)=h(s)$.
\end{enumerate}
\end{lemma}

\section{Open maps}\label{sec:frob}

\paragraph{Locales.} Let us first recall open maps of locales (see \cite{JT}).
A map of locales $p:L\to X$ is open if the image of any open sublocale of $L$ is an open sublocale of $X$. Equivalently, $p$ is open if and only if the inverse image homomorphism $p^*:X\to L$ has a left adjoint $p_!:L\to X$ which is a homomorphism of $X$-modules; that is, such that for all $x\in X$ and $a\in L$ the Frobenius reciprocity condition holds:
\[
p_!(a\wedge p^*(x)) = p_!(a)\wedge x\;.
\]
Then it is also true that a sublocale of $X$ is open if and only if its inclusion into $X$ in $\Loc=\opp\Frm$ is an open map. Moreover, the open sublocales of $X$ can be identified with elements $u\in X$: the quotient that determines the corresponding open sublocale can be taken to be $(-)\wedge u: X\to\downsegment u$. Finally, one important property of open maps is their stability under pullbacks: if the following is a pullback diagram in $\Loc$ and $p$ is open then so is $\pi_1$,
\[
\xymatrix{
Y\otimes_X L\ar[rr]^-{\pi_2}\ar[d]_{\pi_1}&&L\ar[d]^p\\
Y\ar[rr]_{f}&&X
}
\]
and the following diagram is commutative (Beck--Chevalley condition):
\[
\xymatrix{
Y\otimes_X L\ar@{<-}[rr]^-{\pi^*_2}\ar[d]_{{\pi_1}_!}&&L\ar[d]^{p_!}\\
Y\ar@{<-}[rr]_{f^*}&&X
}
\]

\paragraph{Frobenius reciprocity conditions.} A straightforward generalization of the facts above does not exist for an involutive quantale $X$, as there is no a priori known notion of open quantic subspace of $X$, and the elements $u\in X$ do not determine quotients of $X$ in $\Qu$ in a canonical way. So we shall look at maps of involutive quantales satisfying conditions that mimick the Frobenius reciprocity condition of locales, at the same time keeping in mind that any reasonable notion of open map should be such that pullbacks of open maps are open.
Namely, we shall examine the following conditions for a semiopen map $p:Q\to X$ of involutive quantales:

\begin{description}
\item[FR1:] $p_!(a p^*(x)) = p_!(a) x$ for all $x\in X$ and $a\in Q$;
\item[FR2:] $p_!(a p^*(x) b) = p_!(a)x p_!(b)$ for all $x\in X$ and $a,b\in Q$.
\end{description}

The first condition, FR1, which we refer to as the \emph{one-sided Frobenius reciprocity condition}, is an immediate generalization of the Frobenius reciprocity condition of locales, stating that $p_!$ is a homomorphism of left $X$-modules. We shall call a semiopen map that satisfies FR1 \emph{weakly open}, following~\cite{QFB}. We note that, due to the involution, FR1 is equivalent to the analogous condition applied to right $X$-modules:

\begin{lemma}
Let $p:Q\to X$ be a weakly open map. Then for all $a\in Q$ and $x\in X$ we have
\[
p_!(p^*(x)a)=xp_!(a)\;.
\]
\end{lemma}

\begin{proof}
Let $a\in Q$ and $x\in X$. Then $p_!\bigl(p^*(x)a\bigr)=p_!\bigl(p^*(x)a\bigr)^{**}=p_!\bigl((p^*(x)a)^*\bigr)^*=p_!\bigl(a^*p^*(x)^*\bigr)^*
=p_!\bigl(a^*p^*(x^*)\bigr)^*
=\bigl(p_!(a^*)x^*\bigr)^* = \bigl(p_!(a)^*x^*\bigr)^*=x p_!(a)$.
\end{proof}

Another simple property of weakly open maps is the following:

\begin{lemma}\label{lem:wos}
Let $p:Q\to X$ be a weakly open map with $X$ a unital involutive quantale. Then $p$ is a surjection if and only if $p_!(p^*(e))=e$.
\end{lemma}

\begin{proof}
$p$ is a surjection if and only if for all $x\in X$ we have $p_!(p^*(x))=x$, so one implication is trivial. Let then $p_!(p^*(e))=e$ and assume that $p$ is weakly open. Then for all $x\in X$ we have
$p_!(p^*(x))=p_!(p^*(ex))=p_!(p^*(e)p^*(x))=p_!(p^*(e))x=ex=x$.
\end{proof}

Contrary to the situation with locales, it is not to be expected that weakly open maps should be stable under pullbacks in $\QS$ (\cf\ section~\ref{sec:stability}).  Hence, as a working definition of openness, even if a naive one, let us say that a semiopen map $p:Q\to X$ is \emph{open} if for all maps $f:Y\to X$ the pullback $f^*(p)$ in $\QS$ of $p$ along $f$ is weakly open:
\[
\xymatrix{
P\ar[rr]\ar[d]_{f^*(p)}&&Q\ar[d]^p\\
Y\ar[rr]_{f}&&X
}
\]
It is clear that identity maps are open, and that the class of open maps is closed under composition, so the open maps define a subcategory of $\QS$. And, 
since any pullback of $f^*(p)$ is itself isomorphic to a pullback of $p$, this subcategory is closed under pullbacks along arbitrary maps:

\begin{lemma}
Let $p:Q\to X$ be an open map of involutive quantales, and let $f:Y\to X$ be an arbitrary map of involutive quantales. Then the pullback $f^*(p)$ is an open map.
\end{lemma}

Note that it is not implied that open maps of locales (in the usual sense) are necessarily open when regarded as maps in $\QS$.

The second condition, FR2, will be referred to as the \emph{two-sided Frobenius reciprocity condition}. It is not a generalization of the Frobenius reciprocity condition of locales, for in general it is not satisfied by open maps of locales, as the following shows:

\begin{lemma}
Let $p:L\to X$ be a semiopen map of locales that satisfies FR2. Then for all $a,b\in L$ we have $p_!(a\wedge b)=p_!(a)\wedge p_!(b)$.
\end{lemma}

\begin{proof}
Let $a,b\in L$. Then $p_!(a\wedge b)=p_!(a\wedge 1_L \wedge b)=p_!(a\wedge p^*(1_X)\wedge b)
=p_!(a)\wedge 1_X \wedge p_!(b)=p_!(a)\wedge p_!(b)$.
\end{proof}

\begin{example}
Let $f:X\to Y$ be a continuous open map of topological spaces, with $X$ Hausdorff. The map of locales $\topology(f):\topology(X)\to\topology(Y)$ defined by $\topology(f)^*=f^{-1}$ satisfies FR2 if and only if $f$ is injective.
\end{example}

This example also shows that in general FR1 does not imply FR2, whereas in some situations (in particular those coming from Fell bundles --- see below) FR2 implies FR1:

\begin{lemma}
Let $p:Q\to X$ be a semiopen map of involutive quantales satisfying FR2.
If $X$ is unital and $p_!(p^*(e))=e$ then $p$ is surjective and it satisfies FR1.
\end{lemma}

\begin{proof}
Let $e$ be the multiplicative unit of $X$. For all $a\in Q$ and $x\in X$ we have, using FR2,
\begin{eqnarray*}
p_!(ap^*(x))&=&p_!(ap^*(xe))=p_!(ap^*(x)p^*(e))\\
&=&p_!(a)xp_!(p^*(e))=p_!(a)xe=p_!(a)x\;,
\end{eqnarray*}
and thus FR1 holds. Surjectivity of $p$ follows from Lemma~\ref{lem:wos}.
\end{proof}

As we shall see in section~\ref{sec:stability}, the conjunction of FR1 and FR2 is interesting, at least in the case of surjective maps, because the class of semiopen surjections satisfying FR1 and FR2 is closed under pullbacks. In particular, it follows that such surjections are examples of open maps according to the definition above.

\paragraph{Fell bundles.} A Fell bundle $\pi:E\to G$ on a (suitable) topological \'etale groupoid $G$ is a Banach bundle on $G$ equipped with additional structure such that, among other things, $E$ is an involutive semicategory and $p$ is functorial~\cite{Kumjian98}. Associated to a Fell bundle $\pi:E\to G$ there is a convolution algebra of sections $C_c(G,E)$ (this generalizes the usual convolution algebra $C_c(G)$ of continuous compactly supported functions $G\to\CC$) and for a large class of C*-completions $A$ of $C_c(G,E)$ we obtain maps of involutive quantales $p:\Max A\to\topology(G)$, where $\Max A$ and $\topology(G)$ are the involutive quantales associated to $A$ and $G$, respectively ($\Max A$ consists of all the norm-closed linear subspaces of $A$ with multiplication given by the topological closure of the linear span of the pointwise multiplication, and $\topology(G)$ is the topology of $G$ under pointwise multiplication). The properties of $p$ are closely related to properties of $\pi$ and $A$. In particular, the situations where $p$ is semiopen are closely related to $G$ being Hausdorff and $A$ being the reduced C*-algebra $C_r^*(G,E)$. Further imposing on $p$ conditions that approach FR2 has the effect of restricting the bundle to be a line bundle, or even force $G$ to be a principal groupoid (an equivalence relation). See~\cite{QFB}.

\begin{example}
As a simple illustration, take $G$ to be the discrete pair groupoid (= total binary relation) on the set $\{1,\ldots,n\}$, and let $\pi:=\pi_1:G\times\CC\to G$. The convolution algebra of $\pi$ can be identified with the matrix algebra $A=M_n(\CC)$, and the quantale $\topology(G)$ is the quantale of binary relations on $\{1,\ldots,n\}$. For each binary relation $U\in\topology(G)$ let $p^*(U)$ be the set of matrices $M$ such that $m_{ij}=0$ whenever $(i,j)\notin U$. Then $p^*(U)\in\Max A$. The mapping $p^*:\topology(G)\to\Max A$ is an injective homomorphism of involutive quantales, and it has a left adjoint $p_!$ which to each linear subspace $V\subseteq A$ assigns the relation
\[
\{(i,j)\in G\st m_{ij}\neq 0\textrm{ for some }M\in V\}\;.
\]
The semiopen map $p$ thus defined is a surjection in $\QS$. It satisfies FR2, and therefore also FR1 because $\topology(G)$ is unital.
\end{example}

\begin{example}
Let again $\pi=\pi_1:G\times\CC\to G$, now for $G$ a non-trivial finite discrete group. Then the convolution algebra $A$ can be identified with the group algebra $\CC G$, and again we obtain a surjective semiopen map $p:\Max A\to\topology(G)$, such that $p^*(U)=\CC U$ for all $U\subseteq G$ and
\[
p_!(V)=\{g\in G\st v_g\neq 0\textrm{ for some }v\in V\}
\]
for all $V\in\Max A$. Now $p$ satisfies FR1 but not FR2.
\end{example}

\section{Pullbacks of quantic spaces}

Given two sup-lattices $L$ and $M$ we write $L\otimes M$ for their tensor product, as in \cite{JT}, and $L\oplus M$ ($=L\times M$) for their coproduct, which we shall refer to as the \emph{direct sum} of $L$ and $M$. Similarly, $\bigoplus_i L_i$ is the coproduct of a family $(L_i)$ of sup-lattices.

\paragraph{Products.} Let $Y$ and $Q$ be involutive quantales. The product $Y*Q$ in $\QS$ can be constructed concretely as being the following sup-lattice:
\begin{eqnarray*}
Y*Q &:=& \bigoplus_{n=1}^\infty T_n\\
&\cong&Y\oplus Q\oplus (Y\otimes Q)\oplus (Q\otimes Y)\oplus (Y\otimes Q\otimes Y)\oplus (Q\otimes Y\otimes Q)\oplus \cdots\;,
\end{eqnarray*}
where
\begin{eqnarray*}
T_{4k+1} &=& Y\otimes (Q\otimes Y)^{\otimes k}\\
T_{4k+2} &=& Q\otimes (Y\otimes Q)^{\otimes k}\\
T_{4k+3} &=& (Y\otimes Q)^{\otimes (k+1)}\\
T_{4k+4} &=& (Q\otimes Y)^{\otimes (k+1)}
\end{eqnarray*}
for all $k\in\{0,1,2,\ldots\}$ and we use the notation $Z^{\otimes k}$ for the tensor product $Z\otimes\cdots\otimes Z$ of $k$ copies of $Z$, with $Z^{\otimes 0}$ being $\Omega=\mathcal P(\{*\})$.
Hence,
\[
(Y*Q)\otimes(Y*Q)\cong \bigoplus_{m,n=1}^\infty T_m\otimes T_n\;.
\]
The quantale multiplication on $Y*Q$ is given by a sup-lattice homomorphism
\[
(Y*Q)\otimes(Y*Q)\to Y*Q\;,
\]
which, due to the universal property of the sup-lattice coproduct, is equivalent to specifying, for each $m$ and $n$, a sup-lattice homomorphism
\[
\mu_{m,n}:T_m\otimes T_n\to Y*Q\;.
\]
Succintly, the homomorphisms $\mu_{m,n}$ are computed using the following rules, for all $y,y'\in Y$ and $a,a'\in Q$:
\begin{eqnarray*}
(\cdots\otimes y)(a\otimes\cdots) &=& \cdots\otimes y\otimes a\otimes \cdots\\
(\cdots\otimes a)(y\otimes\cdots) &=& \cdots\otimes a\otimes y\otimes \cdots\\
(\cdots\otimes y)(y'\otimes\cdots) &=& \cdots\otimes yy'\otimes \cdots\\
(\cdots\otimes a)(a'\otimes\cdots) &=& \cdots\otimes aa'\otimes \cdots
\end{eqnarray*}
For instance,
\[
\mu_{5,6}:(Y\otimes (Q\otimes Y))\otimes (Q\otimes (Y\otimes Q))\to Y*Q
\]
is given by
\[
(y\otimes (a\otimes y'))\otimes(b\otimes (z\otimes b'))\mapsto (y\otimes a)\otimes (y'\otimes b)\otimes (z\otimes b')\;,
\]
which means $\mu_{5,6}$ is an isomorphism
\[
T_5\otimes T_6\stackrel\cong\to T_{11}
\]
composed with the inclusion $T_{11}\to Y*Q$.
Another example, now one that uses the multiplication of $Y$, is
\[
\mu_{5,5}:(Y\otimes (Q\otimes Y))\otimes (Y\otimes (Q\otimes Y))\to Y*Q\;,
\]
which is given by
\[
(y\otimes (a\otimes y'))\otimes(z\otimes (b\otimes z'))\mapsto y\otimes (a\otimes y'z)\otimes (b\otimes z')\;,
\]
so $\mu_{5,5}$ is a homomorphism
\[
T_5\otimes T_5\to T_9
\]
following by the inclusion $T_9\to Y*Q$. We note that this homomorphism is well defined because the distributivity of the quantale product of $Y$ ensures that $\mu_{5,5}$, regarded as a multilinear map, preserves joins in each variable separately.
A complete specification of the multiplication involves sixteen different cases, corresponding to the four types of direct summands in the construction of $Y*Q$, and we omit it. We note that, similarly to the two examples just seen, for each $m$ and $n$ there is $k$ such that the homomorphism
$\mu_{m,n}$ is the composition of a homomorphism $T_m\otimes T_n\to T_k$ with the inclusion $T_k\to Y*Q$.

Due to the universal property of the tensor product the multiplication thus obtained on $Y*Q$ preserves joins in each variable. It is also associative due to the associativity of both $Y$ and $Q$, and due to the fact that if a product of pure tensors $\boldsymbol\tau\otimes(\boldsymbol\tau'\otimes \boldsymbol\tau'')$ is valued in $T_n$ then so is $(\boldsymbol\tau\otimes\boldsymbol\tau')\otimes \boldsymbol\tau''$, and thus no bracketing mismatches occur. Therefore $Y*Q$ is a quantale. It is also an involutive quantale, with the involution $Y*Q\to Y*Q$ defined again separately for each $T_n$ by the generic rule
\[
(\cdots a\otimes y\otimes b\otimes z \cdots)^*\quad =\quad \cdots z^*\otimes b^*\otimes y^*\otimes a^*\cdots\;.
\]
The embeddings $Y=T_1\to Y*Q$ and $Q=T_2\to Y*Q$ are therefore homomorphisms of involutive quantales, and they provide the projections
\[
\xymatrix{
Y&&Y*Q\ar[ll]_{\pi_1}\ar[rr]^{\pi_2}&&Q
}
\]
of the product in $QS$.
The pairing $\langle f,g\rangle$ of two maps $f$ and $g$
\[
\xymatrix{
&&R\ar[ddll]_f\ar[ddrr]^g\ar@{.>}[dd]|{\langle f,g\rangle}\\
~\\
Y&&Y*Q\ar[ll]^{\pi_1}\ar[rr]_{\pi_2}&&Q
}
\]
is again defined separately on each $T_n$; that is, its inverse image $\langle f,g\rangle^*$ is obtained by linear extension from the assignments
\[
\xymatrix{
\cdots Y\otimes Q\otimes Y\otimes Q\cdots\ar[rr]_-{\langle f,g\rangle^*=[f^*,g^*]}&& R\\
\cdots y\otimes a\otimes y'\otimes a'\cdots\ar@{|->}[rr]&&\cdots f^*(y)g^*(a)f^*(y')g^*(a')\cdots
}
\]
It is straightforward to verify that $\langle f,g\rangle$ is a map of involutive quantales, and that it is the unique map making the above diagram commute, so $Y*Q$ is a product in $QS$ as intended.

\paragraph{Pullbacks.} Let the following be a pullback diagram in $\QS$:
\[
\xymatrix{
Y*_X Q\ar[rr]^{\pi_2}\ar[d]_{\pi_1}&&Q\ar[d]^p\\
Y\ar[rr]_{f}&&X
}
\]
Of course, $Y*_X Q$ equals $(Y*Q)_j$ where $j$ is the least quantic nucleus on $Y*Q$ such that $j(p^*(x))=j(f^*(x))$ for all $x\in X$ (for notational convenience we shall usually identify $f^*(x)$ with $\pi_1^*(f^*(x))$ and $p^*(x)$ with $\pi_2^*(p^*(x))$). By Lemma~\ref{relations}, taking $\widetilde R$ to be the closure of $R$ under involution and multiplication by elements of $Y*Q$, we obtain
\begin{eqnarray*}
Y*_X Q&=&\{\alpha\in Y*Q\st\forall_{(r,s)\in \widetilde R}\ \ r\le \alpha\shortiff s\le \alpha\}\;.
\end{eqnarray*}
More explicitly, and taking into account that in this case $R$ is already closed under involution, $Y*_X Q$ consists of those $\alpha\in Y*Q$ such that for all $x\in X$ and all $z,w\in Y*Q$ the following conditions hold:
\begin{eqnarray}
p^*(x)\le \alpha&\iff& f^*(x)\le \alpha\label{pbrels1}\\
zp^*(x)\le \alpha&\iff& zf^*(x)\le \alpha\\
p^*(x)w\le \alpha &\iff&f^*(x)w\le \alpha\\
zp^*(x)w\le \alpha&\iff& zf^*(x)w\le \alpha\label{pbrels4}
\end{eqnarray}
Moreover, it suffices to take $z$ and $w$ to be pure tensors
\[
\cdots y\otimes a\otimes y'\otimes a'\cdots\;,
\]
and thus $Y*_X Q$ consists of those $\alpha\in Y* Q$ that satisfy the following nine types of conditions, for all $a,a'\in Q$ and all $y,y'\in Y$, and all pure tensors $\boldsymbol \tau,\boldsymbol \tau'\in Y*Q$ that are appropriate in the sense that they yield $\otimes$-strings $\boldsymbol\tau \otimes a$, $y\otimes\boldsymbol\tau'$, etc., that alternate the elements of $Q$ and $Y$:

\begin{eqnarray}
p^*(x)\le \alpha&\iff& f^*(x)\le \alpha\label{ninetypes1}\\
p^*(x)a\otimes \boldsymbol \tau\ \le \alpha &\iff& f^*(x)\otimes a\otimes \boldsymbol \tau\ \le \alpha\label{ninetypes2}\\
p^*(x)\otimes y \otimes \boldsymbol \tau\ \le \alpha &\iff& f^*(x)y\otimes \boldsymbol \tau\ \le \alpha\label{ninetypes3}\\
\boldsymbol \tau \otimes ap^*(x)\le \alpha &\iff& \boldsymbol \tau \otimes a\otimes f^*(x)\le \alpha\label{ninetypes4}\\
\boldsymbol \tau \otimes y\otimes p^*(x)\le \alpha &\iff& \boldsymbol \tau \otimes yf^*(x)\le \alpha\label{ninetypes5}\\
\boldsymbol \tau \otimes ap^*(x)a'\otimes \boldsymbol \tau'\ \le \alpha &\iff& \boldsymbol \tau \otimes a\otimes f^*(x)\otimes a'\otimes \boldsymbol \tau'\ \le \alpha\label{ninetypes6}\\
\boldsymbol \tau \otimes y\otimes p^*(x)a\otimes \boldsymbol \tau'\ \le \alpha &\iff& \boldsymbol \tau \otimes yf^*(x)\otimes a\otimes \boldsymbol \tau'\ \le \alpha\label{ninetypes7}\\
\boldsymbol \tau\otimes ap^*(x)\otimes y \otimes \boldsymbol \tau'\ \le \alpha &\iff& \boldsymbol \tau \otimes a\otimes f^*(x)y\otimes \boldsymbol \tau'\ \le \alpha\label{ninetypes8}\\
\boldsymbol \tau \otimes y\otimes p^*(x)\otimes y'\otimes \boldsymbol \tau'\ \le \alpha &\iff& \boldsymbol \tau \otimes yf^*(x)y'\otimes \boldsymbol \tau'\ \le \alpha\label{ninetypes9}
\end{eqnarray}

\section{Stability under pullbacks}\label{sec:stability}

\begin{theorem}\label{thm:pbstab}
Let $p:Q\to X$ be a semiopen surjective map of involutive quantales satisfying both FR1 and FR2, and let the following be a pullback diagram in $\QS$:
\begin{equation}\label{pullback}
\xymatrix{
Y*_X Q\ar[rr]^{\pi_2}\ar[d]_{\pi_1}&&Q\ar[d]^p\\
Y\ar[rr]_{f}&&X
}
\end{equation}
Then $\pi_1$ is a semiopen surjection satisfying both FR1 and FR2, and
the following diagram in the category of sup-lattices is commutative (Beck--Chevalley condition):
\begin{equation}\label{beckchevalley}
\xymatrix{
Y*_X Q\ar@{<-}[rr]^{\pi^*_2}\ar[d]_{{\pi_1}_!}&&Q\ar[d]^{p_!}\\
Y\ar@{<-}[rr]_{f^*}&&X
}
\end{equation}
\end{theorem}

\begin{proof}
Consider the pullback diagram \eqref{pullback} with $p$ a semiopen surjection satisfying FR1 and FR2. In order to show that $\pi_1$ is semiopen we begin by defining a sup-lattice homomorphism $Y*_X Q\to Y$ which will then be shown to be the required direct image homomorphism of $\pi_1$. First let us recall the following notation from the definition of the product in $\QS$, where $k\in\{0,1,2,\ldots\}$:
\begin{eqnarray*}
Y*Q&=&\bigoplus_{n=1}^\infty T_n\\
T_{4k+1} &=& Y\otimes (Q\otimes Y)^{\otimes k}\\
T_{4k+2} &=& Q\otimes (Y\otimes Q)^{\otimes k}\\
T_{4k+3} &=& (Y\otimes Q)^{\otimes (k+1)}\\
T_{4k+4} &=& (Q\otimes Y)^{\otimes (k+1)}\;.
\end{eqnarray*}
For each $n$ define a sup-lattice homomorphism $h_n:T_n\to Y$:
\begin{eqnarray}
h_1(y) &=& y\label{h1condition}\\
h_2(a) &=& f^*(p_!(a))\\
h_3(y\otimes a) &=& yf^*(p_!(a))\\
h_4(a\otimes y) &=& f^*(p_!(a))y
\end{eqnarray}
\[\vdots\]
\begin{eqnarray}
h_n(\cdots y\otimes a\otimes y'\otimes a'\cdots)&=&\cdots yf^*(p_!(a))y'f^*(p_!(a'))\cdots \label{hncondition}\\
h&=&[h_n]:Y*Q\to Y\;: \label{hcondition}
\end{eqnarray}
\[
\xymatrix{
Y*Q\ar@{.>}[r]^-h&Y\\
T_n\ar[ur]_{h_n}\ar[u]
}
\]
We prove that $h$ factors through the surjection $q:Y*Q\to Y*_X Q$ by showing that it satisfies the following nine conditions [\cf\ \eqref{ninetypes1}--\eqref{ninetypes9}] for all $a,a'\in Q$ and all $y,y'\in Y$, and all appropriate alternated pure tensors $\boldsymbol \tau,\boldsymbol \tau'\in Y*Q$ as in \eqref{ninetypes1}--\eqref{ninetypes9}:
\begin{eqnarray}
h(p^*(x))&=& h(f^*(x))\label{ninetypesagain1}\\
h(p^*(x)a\otimes \boldsymbol \tau) &=& h(f^*(x)\otimes a\otimes \boldsymbol \tau)\label{ninetypesagain2}\\
h(p^*(x)\otimes y \otimes \boldsymbol \tau) &=& h(f^*(x)y\otimes\boldsymbol \tau)\label{ninetypesagain3}\\
h(\boldsymbol \tau \otimes ap^*(x)) &=& h(\boldsymbol \tau \otimes a\otimes f^*(x))\label{ninetypesagain4}\\
h(\boldsymbol \tau \otimes y\otimes p^*(x)) &=& h(\boldsymbol \tau \otimes yf^*(x))\label{ninetypesagain5}\\
h(\boldsymbol \tau \otimes ap^*(x)a'\otimes \boldsymbol \tau') &=& h(\boldsymbol \tau \otimes a\otimes f^*(x)\otimes a'\otimes \boldsymbol \tau')\label{ninetypesagain6}\\
h(\boldsymbol \tau \otimes y\otimes p^*(x)a\otimes \boldsymbol \tau') &=& h(\boldsymbol \tau \otimes yf^*(x)\otimes a\otimes \boldsymbol \tau')\label{ninetypesagain7}\\
h(\boldsymbol \tau\otimes ap^*(x)\otimes y \otimes \boldsymbol \tau') &=& h(\boldsymbol \tau \otimes a\otimes f^*(x)y\otimes \boldsymbol \tau')\label{ninetypesagain8}\\
h(\boldsymbol \tau \otimes y\otimes p^*(x)\otimes y'\otimes \boldsymbol \tau') &=&h(\boldsymbol \tau \otimes yf^*(x)y'\otimes\boldsymbol \tau')\;.\label{ninetypesagain9}
\end{eqnarray}
Condition \eqref{ninetypesagain1} is a consequence of surjectivity:
\[
h(p^*(x))=h_2(p^*(x))=f^*(p_!(p^*(x)))=f^*(x)=h_1(f^*(x))=h(f^*(x))\;.
\]
Similarly, \eqref{ninetypesagain3}, \eqref{ninetypesagain5} and \eqref{ninetypesagain9} follow from surjectivity: for instance, for \eqref{ninetypesagain9} we obtain, applying \eqref{h1condition}--\eqref{hcondition},
\begin{eqnarray*}
h(\boldsymbol \tau\otimes y\otimes p^*(x)\otimes y'\otimes\boldsymbol \tau')&=&h(\boldsymbol \tau) yf^*(p_!(p^*(x)))y'h(\boldsymbol \tau')=h(\boldsymbol \tau) yf^*(x)y'h(\boldsymbol \tau')\\
&=&h(\boldsymbol \tau \otimes yf^*(x)y'\otimes\boldsymbol \tau')\;.
\end{eqnarray*}
Again applying \eqref{h1condition}--\eqref{hcondition}, conditions \eqref{ninetypesagain2}, \eqref{ninetypesagain4}, \eqref{ninetypesagain7} and \eqref{ninetypesagain8} follow from FR1: for instance, for \eqref{ninetypesagain2} we have
\begin{eqnarray*}
h(p^*(x)a\otimes\boldsymbol\tau)&=&f^*(p_!(p^*(x)a))h(\boldsymbol\tau)=f^*(xp_!(a))h(\boldsymbol\tau)\\
&=&
f^*(x)f^*(p_!(a))h(\boldsymbol\tau)=
h(f^*(x)\otimes a\otimes\boldsymbol\tau)\;.
\end{eqnarray*}
Finally, still applying \eqref{h1condition}--\eqref{hcondition}, \eqref{ninetypesagain6} follows from FR2:
\begin{eqnarray*}
h(\boldsymbol\tau \otimes ap^*(x)a'\otimes\boldsymbol\tau') &=& h(\boldsymbol\tau) f^*(p_!(ap^*(x)a'))h(\boldsymbol\tau')
=h(\boldsymbol\tau) f^*(p_!(a) x p_!(a'))h(\boldsymbol\tau')\\
&=& h(\boldsymbol\tau) f^*(p_!(a))f^*(x)f^*(p_!(a'))h(\boldsymbol\tau')\\
& = &
h(\boldsymbol\tau \otimes a\otimes f^*(x)\otimes a'\otimes \boldsymbol\tau')\;.
\end{eqnarray*}
Since $h$ respects all the conditions \eqref{ninetypesagain1}--\eqref{ninetypesagain9} it factors through $Y*_X Q$ via a sup-lattice homomorphism $\tilde h$:
\[
\xymatrix{Y*Q\ar@(ur,ul)[rr]^h\ar@{->>}[r]_-q&Y*_X Q\ar[r]_-{\tilde h}&Y}
\]
Now we show that $\tilde h$ is left adjoint to $\pi_1^*$.
The counit of the adjunction is immediate, since $\tilde h(\pi_1^*(y))=y$ for all $y\in Y$. This also shows that $\pi_1$ is a surjection. In order to prove the unit of the adjunction let us use abbreviations such as
\[
\cdots a\qtimes y\qtimes a'\qtimes y'\cdots \ \ \ :=\ \ \ q(\cdots a\otimes y\otimes a'\otimes y'\cdots)\;,
\]
and let us consider a ``word''
\[w\quad  :=\quad  a_1\qtimes y_1\qtimes\cdots\qtimes a_n\qtimes y_n\quad  \in \quad  q(T_{4n})\quad  \subseteq \quad Y*_X Q\;.\]
We have
\begin{eqnarray*}
w&=& a_1\qtimes y_1\qtimes\cdots\qtimes a_n\qtimes y_n\\
&\le&p^*(p_!(a_1))\qtimes y_1\qtimes\cdots\qtimes p^*(p_!(a_n))\qtimes y_n\\
&=&\pi_1^*\bigl(\underbrace{f^*(p_!(a_1)) y_1\cdots f^*(p_!(a_n)) y_n}_{\in T_1=Y}\bigr)\\
&=&\pi_1^*\bigl( h(a_1\otimes y_1\otimes\cdots\otimes a_n\otimes y_n)\bigr)\\
&=&\pi_1^*(\tilde h(w))\;.
\end{eqnarray*}
Similar reasoning applies to any $z_1\qtimes\cdots\qtimes z_m$ with the $z_i's$ taken alternately from $Q$ and $Y$, so we conclude that
\[
\alpha\le \pi_1^*\bigl(\tilde h(\alpha)\bigr)
\]
for all $\alpha\in Y*_X Q$, thus showing that $\tilde h$ is left adjoint to $\pi_1^*$, so $\pi_1$ is semiopen.
It is also clear that $\tilde h$ satisfies FR1 because for all $a,a'\in Q$ and $y,y'\in Y$, with $\boldsymbol\tau$ being an image by $q$ of an appropriate pure tensor in $Y*Q$, we have
\begin{eqnarray*}
\tilde h\bigl(\underbrace{y\pi_1^*(y')}_{\in T_1=Y}\bigr)&=&\tilde h(yy')\\
&=&\tilde h(y)y'\;,
\end{eqnarray*}
\begin{eqnarray*}
\tilde h\bigl((\boldsymbol\tau \qtimes y)\pi_1^*(y')\bigr)&=&\tilde h(\boldsymbol\tau\qtimes yy')\\
&=&\tilde h(\boldsymbol\tau) yy'\\
&=&\tilde h(\boldsymbol\tau \qtimes y)y'\;,
\end{eqnarray*}
and
\begin{eqnarray*}
\tilde h\bigl(a\pi_1^*(y)\bigr)&=&\tilde h(a\qtimes y)\\
&=&f^*(p_!(a))y\\
&=&\tilde h(a)y\;,
\end{eqnarray*}
\begin{eqnarray*}
\tilde h\bigl((\boldsymbol\tau\qtimes a)\pi_1^*(y)\bigr)&=&\tilde h(\boldsymbol\tau \qtimes a\qtimes y)\\
&=&\tilde h(\boldsymbol\tau) f^*(p_!(a))y\\
&=&\tilde h(\boldsymbol\tau \qtimes a)y\;.
\end{eqnarray*}
FR2 is proved in a similar way, now computing
$
\tilde h\bigl([\boldsymbol\tau\qtimes] z\pi_1^*(y)z'[\qtimes\boldsymbol\tau']\bigr)
$
for a total of sixteen combinations with $z,z'\in Q\cup Y$. To conclude, the commutativity of the diagram \eqref{beckchevalley} is just the statement that $\tilde h(a)=f^*(p_!(a))$ for all $a\in Q$.
\end{proof}

\begin{corollary}
Any weakly open surjection satisfying FR2 is an open map.
\end{corollary}

\vspace*{5mm}
\noindent {\sc
Centro de An\'alise Matem\'atica, Geometria e Sistemas Din\^amicos
Departamento de Matem\'{a}tica, Instituto Superior T\'{e}cnico\\
Universidade de Lisboa\\
Av.\ Rovisco Pais 1, 1049-001 Lisboa, Portugal}\\
{\it E-mail:} {\sf pmr@math.tecnico.ulisboa.pt}

\end{document}